\documentstyle[amssymb,amsfonts]{amsart}

\def\N{{\hbox{\bf N}}}

 at 10 true pt

\def\Re{\operatorname{Re}}

\newenvironment{proof}{\noindent {\bf Proof} }{\endprf\par}
\def \endprf{\hfill  {\vrule height6pt width6pt depth0pt}\medskip}
\def\emph#1{{\it #1}}
\def\textbf#1{{\bf #1}}

\theoremstyle{plain}
  \newtheorem{theorem}[subsection]{Theorem}

  \newtheorem{lemma}[subsection]{Lemma}

\theoremstyle{remark}

\theoremstyle{definition}
  \newtheorem{definition}[subsection]{Definition}

\include{psfig}

\begin{document}

\title[The prime divisors
contain arbitrary large truncated classes]{The prime divisors in every class contain arbitrary large truncated
classes}

\author{Chunlei Liu}
\address{Department of Mathematics, Shanghai Jiao Tong University, Shanghai, 200240}
\email{clliu@@sjtu.edu.cn}

\begin{abstract}  Let ${\mathbb F}_q$ be the finite field with $q$ elements, and $\hat{C}$ a projective curve
over ${\mathbb F}_q$. We show that the prime divisors on $\hat{C}$
in every class contain arbitrary large truncated generalized classes
of finite effective divisors.\end{abstract}

\maketitle

\section{Introduction}
 Green-Tao \cite{gt-ap} proved that the primes contains arbitrary long
arithmetic progression. Thai \cite{thai} proved a polynomial analog
of the above result. In this paper we shall prove a more general geometric
analog.

Let ${\mathbb F}_q$ be the finite field with $q$ elements, $\hat{C}$
a projective curve over ${\mathbb F}_q$, and $K$ the function field
of $\hat{C}$. We fix an element $t\in K$ which is transcendental
over ${\mathbb F}_q$. \begin{definition}If $P$ is a pole of $t$,
then we call $P$ a point at infinity, and write
$P\mid\infty$.\end{definition}Let $C$ be the finite part of
$\hat{C}$.
\begin{definition} For each nonzero function $f\in K$, the divisor of $f$ on $C$ is defined to be
$$(f):=\sum_{P\in C}{\rm ord}_P(f)\cdot P.$$
\end{definition}

\begin{definition}For every divisor $D$ on $C$, we write
$$L(D):=\{f\in K^{\times}\mid (f)+D\geq0\}\cup\{0\}.$$
\end{definition}
\begin{definition} Let $M,D_1,D_2$ be effective divisors on $C$. If there is a function $f\in K^{\times}$ such that
$f-1\in L(D_1-M)$ and $D_2=(f)+D_1$, then $D_2$ is said to be equivalent to $D_1$ modulo $M$.
\end{definition}
\begin{definition}Let $M$ and $D$ be two divisors on $C$ such that $D\geq
M$. Let $a\in L(D)$ and $r>0$. We call
$$\{f\in L(D)\mid f-a\in L(M), {\rm ord}_{\infty}(f-a)>-r\}$$
a truncated residue class of $L(D)$, where $${\rm
ord}_{\infty}(f)=\min_{P\mid\infty} \{{\rm ord}_P(f)\}.$$ The truncated class is called principal if $M$ is principal.
\end{definition}
Note that the function ${\rm ord}_{\infty}(\cdot)$ extends to
$$K_{\infty}:=\prod_{P\mid\infty}K_P,$$ where $K_P$ is the
completion of $K$ at $P$, and $K$ is embedded into $K_{\infty}$
canonically.

\begin{definition}Let $D$ be a divisor on $C$, and let $A$ be a truncated residue class of $L(D)$. We call
$$\{(f)+D\mid f\in A\}$$ a
truncated generalized class of effective divisors on $C$.
\end{definition}
 In this paper we shall prove the following
generalization of the result of Thai in \cite{thai}.

\begin{theorem}\label{main}The prime divisors on $C$ in every equivalence class contain arbitrary large truncated
equivalence classes of effective divisors.\end{theorem}

 A positive
density version of the above theorem can be proved similarly.

\section{Pseudo-random measures on inverse systems}
In this section we establish the relationship between two kinds of measures on inverse systems.

Let $k$ be a fixed positive integer, $D$ a fixed nonzero effective
divisor on $C$, and $I$ the set of polynomials in ${\mathbb F}_q[t]$
which are prime to every nonzero   function in $O_C\cap B_{k}$,
where $O_C$ is the ring of regular functions on $C$, and
$$B_k:=\{f\in K_{\infty}\mid {\rm ord}_{\infty}(f)>-k\}.$$ Then
$\{L(D)/(NL(D))\}_{N\in I}\}$ is an inverse system of finite groups.
For each $j\in O_C\cap B_k$, we write $e_j=(O_C\cap
B_k)\setminus\{j\}$. Then $(O_C\cap B_k,\{e_j\}_{j\in O_C\cap B_k})$
is a hyper-graph. To each edge $e_j$, we associate the system $\{(
L(D) /N L(D) )^{e_j}\}_{N\in I}$. Thus the system $\{( L(D) /N L(D)
)^{e_j}\}_{N\in I,j\in O_C\cap B_k}$ maybe regarded as an inverse
system on the hyper-graph $(O_C\cap B_k,\{e_j\}_{j\in O_C\cap
B_k})$. For each $j\in O_C\cap B_k$, and for each $N\in I$, let
$\tilde{\nu}_{N,j}$ be a nonnegative function on $( L(D) /N L(D)
)^{e_j}$.
\begin{definition}The system
$\{\tilde{\nu}_{N,j}\}_{N\in I,j\in O_C\cap B_k}$ is
called a \emph{pseudo-random} system of measures on the system $\{ L(D) /(N L(D) )\}_{N\in I, j\in O_C\cap B_k}$ if the following conditions are satisfied.\begin{enumerate}
                                    \item For all $j\in O_C\cap B_k$, and for all $\Omega_j\subseteq\{0,1\}^{e_j}\setminus\{0\}$, one has
$$ \frac{1}{ q^{   |e_j|[K: {\mathbb F}_q(t)  ]\deg N}}\sum_{x^{(1)} \in ( L(D) /N L(D) )^{e_j}}\prod_{\omega \in \Omega_j} \tilde{\nu}_{N,j}(x^{(\omega)})  = O(1),$$
uniformly for all $x^{(0)}\in ( L(D) /N L(D) )^{e_j}$
                                    \item Given any choice $\Omega_j\subseteq\{0,1\}^{e_j}$ for each $j\in O_C\cap B_k$, one has
$$
\frac{1}{ q^{   2|O_C\cap B_k|[K: {\mathbb F}_q(t)  ]\deg N}}\sum_{x^{(0)}, x^{(1)} \in( L(D) /N L(D) )^{O_C\cap B_k}}\prod_{j\in O_C\cap B_k} \prod_{\omega \in \Omega_j} \tilde{\nu}_{N,j}(x^{(\omega)})  = 1 + o(1),
'$$as $  \deg N \to \infty $.
                                    \item For all $j\in O_C\cap B_k$, for all $i \in e_j$, for all $\Omega_j\subseteq \{0,1\}^{e_j}$, and for all $M \in{\mathbb N}$, we have
$$
\frac{1}{ q^{   2(|e_j|-1)[K: {\mathbb F}_q(t)  ]\deg N}}\sum_{x^{(0)}, x^{(1)} \in ( L(D) /N L(D) )^{e_j \backslash \{i\}} }{\rm auto}(x,\tilde{\nu}_{N,j})^M= O(1),
$$
where $${\rm auto}(x,\tilde{\nu}_{N,j}):=\frac{1}{ q^{  2[K: {\mathbb F}_q(t)  ]\deg N }}\sum_{ x_i^{(0)}, x_i^{(1)} \in L(D) /N L(D) }\prod_{\omega \in \Omega_j} \tilde{\nu}_{N,j}( x^{(\omega)} ).$$
                                  \end{enumerate}
\end{definition}
For each positive
integer $N\in I$, let $\tilde{\nu}_N$ be a nonnegative function on $ L(D) /(N L(D) )$. \begin{definition}The system
$\{\tilde{\nu}_N\}$ is said to satisfy the $k$-cross-correlation
condition if, given any positive integers $s\leq |O_C\cap B_k|2^{|O_C\cap B_k|},m\leq
2|O_C\cap B_k|$, and given any mutually independent linear forms
$\psi_1,\cdots,\psi_{s}$ in $m$ variables whose coefficients are
  functions in $O_C\cap B_{k}$, we have
$$\frac{1}{ q^{   m[K: {\mathbb F}_q(t)  ]\deg N}}\sum_{\stackrel{x_i\in  L(D) /(N L(D) )}{i=1,\cdots,m}}
\prod_{j=1}^{s}\tilde{\nu}_N(\psi_j(x)+b_j)=1+o(1),\
N\rightarrow\infty$$ uniformly for all $b_1,\cdots,b_{s}\in L(D)/(N L(D) )$.
\end{definition}
\begin{definition}The system
$\{\tilde{\nu}_N\}$ is said to satisfy the $k$-auto-correlation
condition if, given any positive integers $s\leq |O_C\cap
B_k|2^{|O_C\cap B_k|}$, there exists a system $\{\tilde{\tau}_N\}$ of
nonnegative functions  on $\{ L(D) /(N L(D) )\}$ which obeys
the moment condition
$$\frac{1}{ q^{   [K: {\mathbb F}_q(t)  ]\deg N}}\sum_{x\in L(D) /(N L(D) )}\tilde{\tau}_N^M(x)=O_{M,s}(1),\ \forall M\in{\mathbb N}$$ such that
$$\frac{1}{ q^{   [K: {\mathbb F}_q(t)  ]\deg N}}
\sum\limits_{x\in  L(D) /(N L(D) )}\prod_{i=1}^s
 \tilde{\nu}_N(x +
y_i)\leq\sum_{1\leq i<j\leq
s}\tilde{\tau}_N(y_i-y_j).$$\end{definition}\begin{definition}The
system $\{\tilde{\nu}_N\}$ is called a $k$-pseudo-random system of
measure on the inverse system $\{ L(D) /(N L(D) )\}$ if it satisfies
 the $k$-cross-correlation condition and the $k$-auto-correlation
condition.\end{definition}From now on we assume that
$$\tilde{\nu}_{N, j}(x) := \tilde{\nu}_N(\sum_{i\in e_j}(i-j)x_i).$$\begin{theorem}\label{inversesys}If
$\{\tilde{\nu}_N\}$ is $k$-pseudo-random, then
$\{\tilde{\nu}_{N,j}\}$ is pseudo-random.
\end{theorem}\begin{proof}First, we show that, for all $j\in O_C\cap B_k$, and for all $\Omega_j\subseteq\{0,1\}^{e_j}\setminus\{0\}$,
$$ \frac{1}{ q^{   |e_j|[K: {\mathbb F}_q(t)  ]\deg N}}\sum_{x^{(1)} \in ( L(D) /N L(D) )^{e_j}}\prod_{\omega \in \Omega_j} \tilde{\nu}_{N,j}(x^{(\omega)})  = O(1),$$ uniformly for all $x^{(0)}\in ( L(D) /N L(D) )^{e_j}$.
For each $\omega\in\Omega_j$, set
$$\psi_{\omega}(x^{(1)})=\sum_{i\in e_j,\omega_i=1}(i-j)x_i^{(1)},
$$
and $$b_{\omega}=\sum_{i \in e_j,\omega_i=0}(i-j)x_i^{(0)}).$$Then
$$ \frac{1}{ q^{   |e_j|[K: {\mathbb F}_q(t)  ]\deg N}}\sum_{x^{(1)} \in ( L(D) /N L(D) )^{e_j}}\prod_{\omega \in \Omega_j} \tilde{\nu}_{N,j}(x^{(\omega)})$$
$$ =\frac{1}{ q^{   |e_j|[K: {\mathbb F}_q(t)  ]\deg N}}\sum_{x^{(1)} \in ( L(D) /N L(D) )^{e_j}}\prod_{\omega \in \Omega_j} \tilde{\nu}_{N}(\psi_{\omega}(x^{(1)})+b_{\omega})  = O(1).$$

Secondly, we show that, given any choice $\Omega_j\subseteq\{0,1\}^{e_j}$ for each $j\in O_C\cap B_k$,
$$
\frac{1}{ q^{   2|O_C\cap B_k|[K: {\mathbb F}_q(t)  ]\deg N}}\sum_{x^{(0)}, x^{(1)} \in( L(D) /N L(D) )^{O_C\cap B_k}}\prod_{j\in O_C\cap B_k} \prod_{\omega \in \Omega_j} \tilde{\nu}_{N,j}(x^{(\omega)})$$$$  = 1 + o(1),\
\deg N \to \infty.$$
For each pair $(j,\omega)$ with $j\in O_C\cap B_k$ and $\omega\in\Omega_j$, set
$$  \psi_{(j,\omega)}(x)=\sum_{\stackrel{i\in O_C\cap B_k,\delta=0,1}{\omega_i=\delta}}(i-j)x_i^{(\delta)}.
$$
Then $$
\frac{1}{ q^{   2|O_C\cap B_k|[K: {\mathbb F}_q(t)  ]\deg N}}\sum_{x^{(0)}, x^{(1)} \in( L(D) /N L(D) )^{O_C\cap B_k}}\prod_{j\in O_C\cap B_k} \prod_{\omega \in \Omega_j} \tilde{\nu}_{N,j}(x^{(\omega)}) $$$$ =
\frac{1}{ q^{   2|O_C\cap B_k|[K: {\mathbb F}_q(t)  ]\deg N}}\sum_{x^{(0)}, x^{(1)} \in( L(D) /N L(D) )^{O_C\cap B_k}}\prod_{j\in O_C\cap B_k} \prod_{\omega \in \Omega_j} \tilde{\nu}_{N}(\psi_{j,\omega}(x))$$$$  = 1 + o(1).
$$

Finally we show that,  for all $j\in O_C\cap B_k$, for all $i \in e_j$, for all $\Omega_j\subseteq \{0,1\}^{e_j}$, and for all $M \in{\mathbb N}$,
$$
\frac{1}{ q^{   2(|e_j|-1)[K: {\mathbb F}_q(t)  ]\deg N}}\sum_{x^{(0)}, x^{(1)} \in ( L(D) /N L(D) )^{e_j \backslash \{i\}} }{\rm auto}(x,\tilde{\nu}_{N,j})^M= O(1).
$$
By Cauchy-Schwartz it suffices to show that, for $a=0,1$, $$
\frac{1}{ q^{   2(|e_j|-1)[K: {\mathbb F}_q(t)  ]\deg N}}\sum_{x^{(0)}, x^{(1)} \in
( L(D) /N L(D) )^{e_j \backslash \{i\}} }
{\rm auto}(x,\tilde{\nu}_{N,j},a)^{2M}= O(1),
$$
where $${\rm auto}(x,\tilde{\nu}_{N,j},a):=\frac{1}{ q^{\deg N   [K:
{\mathbb F}_q(t)  ]}}\sum_{ x_i^{(a)} \in L(D) /N L(D)
}\prod_{\omega \in \Omega_j,\omega_i=a} \tilde{\nu}_{N,j}(
x^{(\omega)} ).$$ For each $\omega\in \Omega_j$ with $\omega_i=a$,
set
$$\psi_{\omega}(x)=\sum_{l \in e_j\setminus\{i\}} (l-j)x_l^{(\omega_l)}.$$Then $$
\frac{1}{ q^{   2(|e_j|-1)[K: {\mathbb F}_q(t)  ]\deg N}}\sum_{x^{(0)}, x^{(1)} \in
( L(D) /N L(D) )^{e_j \backslash \{i\}} }
{\rm auto}(x,\tilde{\nu}_{N,j},a)^{2M}$$
$$
\leq\frac{1}{ q^{   2(|e_j|-1)[K: {\mathbb F}_q(t)  ]\deg N}}\sum_{x^{(0)}, x^{(1)}
\in ( L(D) /N L(D) )^{e_j \backslash \{i\}}
}\sum_{\stackrel{\omega,\omega'\in\Omega_j}{\omega_i=\omega'_i=a}}\tilde{\tau}^{2M}(\psi_{\omega}(x)-\psi_{\omega'}(x))$$$$
=\sum_{\stackrel{\omega,\omega'\in\Omega_j}{\omega_i=\omega'_i=a}}\frac{1}{ q^{  [K:{\mathbb
Q}]\deg N }}\sum_{x \in  L(D) /(N L(D) ) }\tilde{\tau}^{2M}(x)= O(1).$$
\end{proof}
\section{Pseudo-random measures on $L(D)$}
In this section we establish the relationship between measures on
inverse systems and measures on $L(D)$.

Let $A$ be a positive constant. For $r\in {\mathbb N}$, let
$\nu_r\ll r^A$ be a nonnegative function on $ L(D) $.
\begin{definition}The system $\{\nu_r\}$ is said
to satisfy the $k$-cross-correlation condition if, given any open
compact ${\mathbb F}_q[[1/t]]$-module $I$ in $K_{\infty}$, given any
positive integers $s\leq | O_C \cap B_k|2^{| O_C \cap B_k|},m\leq 2|
O_C \cap B_k|$, and given any mutually independent linear forms
$\psi_1,\cdots,\psi_{s}$ in $m$ variables whose coefficients are
 functions in $ O_C \cap B_{k} $, we have
$$\frac{1}{| L(D) \cap(t^{r} I)|^m}\sum_{\stackrel{x_i\in  L(D) \cap(t^{r} I)}{i=1,\cdots,m}}\prod_{j=1}^{s}\nu_r(\psi_j(x)+b_j)=1+o(1),\
r\rightarrow\infty$$ uniformly for all   functions
$b_1,\cdots,b_{s}\in
 L(D) $.
\end{definition}
\begin{definition}The system $\{\nu_r\}$
is said to satisfy the $k$-auto-correlation condition if given any
positive integers $s\leq |O_C\cap B_k|2^{| O_C \cap B_k|}$, there
exists a system $\{\tau_r\}$ of nonnegative function  on $L(D)$ such
that, given any open compact ${\mathbb F}_q[[1/t]]$-module $I$ in
$K_{\infty}$,
$$\frac{1}{|(t^{r} I)\cap L(D) |}\sum_{x\in (t^{r}I)\cap L(D)}\tau_r^M(x)=O_{M}(1),\ r\rightarrow\infty,\ \forall M\in{\mathbb N}$$
and
$$\frac{1}{|(t^{r}I)\cap L(D)|}
\sum\limits_{x\in (t^{r}I)\cap L(D)}\prod_{i=1}^s
 \nu_r(x +
y_i)\leq\sum_{1\leq i<j\leq
s}\tau_r(y_i-y_j).$$\end{definition}\begin{definition}The system
$\{\nu_r\}$ is $k$-pseudo-random if it satisfies
 the $k$-cross-correlation condition and the $k$-auto-correlation
condition.
\end{definition}
Let $\eta_1,\cdots,\eta_n$ be a ${\mathbb F}_q[t]$-basis of $ L(D)
$, and set $$G=\sum_{j=1}^n{\mathbb F}_q[[1/t]]\eta_i\subseteq
K_{\infty}.$$ From on on we assume that, for each $N\in I$,
$$\tilde{\nu}_{N}(x+NL(D))=\nu_{\deg N}(x)\text{ if }x\in t^{\deg N}G.
$$We now prove the following.
\begin{theorem}\label{pseudorandomrelation}If the system $\{\nu_r\}$ is $k$-pseudo-random,
then the system $\{\tilde{\nu}_N\}_{N\in I}$ is also
$k$-pseudo-random.\end{theorem}
\begin{proof}
First we show that, given any positive integers $s\leq | O_C \cap B_k|2^{| O_C \cap B_k|},m\leq
2| O_C \cap B_k|$, and given any mutually independent linear forms
$\psi_1,\cdots,\psi_{s}$ in $m$ variables whose coefficients are
  functions in $ O_C \cap  B_{k} $,
$$\frac{1}{ q^{ m [K:{\mathbb F}_q(t)]\deg N }}\sum_{\stackrel{x_i\in  L(D) /(N L(D) )}{i=1,\cdots,m}}
\prod_{j=1}^{s}\tilde{\nu}_N(\psi_j(x)+b_j)=1+o(1),\
\deg N\rightarrow\infty$$ uniformly for all $b_1,\cdots,b_{s}\in L(D)/(N L(D) )$.
It suffices to show that for any $S'\subset \{1,\cdots,s\}$,
$$\frac{1}{ q^{m [K:{\mathbb F}_q(t)]\deg N }}\sum_{\stackrel{x_i\in t^{\deg N}G\cap L(D)}{i=1,\cdots,m}}
\prod_{j\in S'}(\tilde{\nu}_N(\psi_j(x)+b_j)-1)=o(1),\
N\rightarrow\infty$$ uniformly for all $b_1,\cdots,b_{s}\in L(D)$. Let $c$ be the maximal degree of the coefficients of the matrix of
$\psi$ with respect to $\{\eta_i\}$. Then
$$\frac{1}{ q^{m [K:{\mathbb F}_q(t)]\deg N }}\sum_{\stackrel{x_i\in t^{\deg N}G\cap L(D)}{i=1,\cdots,m}}
\prod_{j\in S'}(\tilde{\nu}_N(\psi_j(x)+b_j)-1)$$$$=\frac{1}{ q^{m
[K:{\mathbb F}_q(t)]\deg N }} \sum_{\stackrel{y_i\in t^{c}G\cap
L(D)}{i=1,\cdots,m}}\sum_{\stackrel{x_i\in t^{\deg N-c}G\cap
L(D)}{i=1,\cdots,m}} \prod_{j\in S'}(\tilde{\nu}_N(\psi_j(x)+t^{\deg
N-c}\psi_j(y)+b_j)-1)$$$$=\frac{1}{ q^{m [K:{\mathbb F}_q(t)]\deg N
}} \sum_{\stackrel{y_i\in t^{c}G\cap
L(D)}{i=1,\cdots,m}}\sum_{\stackrel{x_i\in t^{\deg N-c}G\cap
L(D)}{i=1,\cdots,m}} \prod_{j\in S'}( \nu_{\deg N}
(\psi_j(x)+b'_j)-1)=o(1),$$ where $b'_j\equiv t^{\deg
N-c}\psi_j(y)+b_j({\rm mod} NL(D))$ lies in $t^{\deg N}G$.

Secondly we show that,
 given any positive integers $s\leq | O_C \cap
B_k|2^{| O_C \cap B_k|}$,
$$\frac{1}{ q^{\deg N  [K:{\mathbb F}_q(t)] }}
\sum\limits_{x\in  L(D) /(N L(D) )}\prod_{i=1}^s
 \tilde{\nu}_N(x +
y_i)\leq\sum_{1\leq i<j\leq s}\tilde{\tau}(y_i-y_j),$$ where
$$\tilde{\tau}_N(x+NL(D))=\tau_{\deg N}(x)\text{ if }x\in t^{\deg N}G.$$
In fact, we have
 $$ \frac{1}{ q^{\deg N  [K:{\mathbb F}_q(t)] }}
\sum\limits_{ x \in  L(D) /(N L(D) )}\prod_{i=1}^{s} \tilde{\nu}_N(x
+ y_i)$$
$$=\frac{1}{ q^{\deg N  [K:{\mathbb F}_q(t)] }} \sum\limits_{ x \in t^{\deg N}G}\prod_{i=1}^{s} \nu_{\deg N} (x + y'_i)$$$$\leq
\sum_{1\leq i<j\leq s}\tau_{\deg N}(y'_i-y'_j)$$$$=\sum_{1\leq
i<j\leq s}\tilde{\tau}_N(y_i-y_j),$$ where $y'_i\equiv y_i({\rm
mod}NL(D))$ lies in $t^{\deg N}G$. The theorem is proved.
\end{proof}

\section{The geometric relative Szemer\'edi theorem}
In this section we prove the geometric relative Szemer\'edi theorem.

For each $N\in I$, let $\tilde{A}_N$ be a subset of
$ L(D) /(N L(D) )$. \begin{definition}The upper density of $\{\tilde{A}_N\}$
relative to $\{\tilde{\nu}_N\}$ is defined to be
$$\limsup_{I\ni N\rightarrow\infty}
\frac{\sum_{x\in \tilde{A}_N}\tilde{\nu}_N(x)}{\sum_{x\in  L(D) /(N L(D) )}\tilde{\nu}_N(x)}.$$\end{definition}
The following version of the geometric relative Szemer\'edi theorem follows from a theorem of Tao in \cite{tao:gaussian}.
\begin{theorem} If the system $\{\tilde{\nu}_{N,j}\}$ is
pseudo-random, and $\{\tilde{A}_N\}$ has positive upper density
relative to $\{\tilde{\nu}_N\}$, then there is a subset
$\tilde{A}_N$ and a truncated principal residue class $A$ of $L(D)$ of size
$|O_C\cap B_k|$ such that
$$ A({\rm mod} N L(D) )\subseteq \tilde{A}_N.$$
\end{theorem}The above theorem, along with Theorem \ref{inversesys}, implies the following.
\begin{theorem}\label{rsinversesys} If the system $\{\tilde{\nu}_{N}\}$ is
$k$-pseudo-random, and $\{\tilde{A}_N\}$ has positive upper density
relative to $\{\tilde{\nu}_N\}$, then there is a subset
$\tilde{A}_N$ and a truncated residue class $A$ of $L(D)$ of size
$|O_C\cap B_k|$ such that
$$ A({\rm mod} N L(D) )\subseteq \tilde{A}_N.$$
\end{theorem}
\begin{definition}For $r\in {\mathbb N}$,
let $A_r$ be a subset of $ L(D) \cap B_r$. The upper density of
$\{A_r\}$ relative to $\{\nu_r\}$ is defined to be
$$\limsup_{r\rightarrow\infty}\frac{\sum_{x\in A_r} \nu_{r} (g)}{\sum_{x\in  L(D) \cap B_r} \nu_{r} (x)}.$$\end{definition}
We now prove the following.
\begin{theorem}\label{rs}If $\{ \nu_{r} \}$ is
$k$-pseudo-random, $c$ is a sufficiently large positive constant
depending only on $k$, $C$ and $D$, and $\{A_{r}\cap B_{r-c}\}$ has
positive upper density relative to $\{ \nu_{r} \}$, then there is a
subset $A_r$ that contains a truncated residue class of $L(D)$ of
size $|O_C\cap B_k|$.
\end{theorem}
\proof We have
$$\frac{\sum_{x\in A_{\deg N}\cap
B_{\deg N-k}}\tilde{\nu}_N(x)}{\sum_{x\in  L(D)
/(NL(D))}\tilde{\nu}_N(x)}=\frac{1}{ q^{ [K:{\mathbb F}_q(t)]\deg N
}}\sum_{x\in A_{\deg N}\cap B_{\deg
N-c}}\tilde{\nu}_N(x)+o(1)$$$$=\frac{1}{ q^{  [K:{\mathbb
F}_q(t)]\deg N }}\sum_{x\in A_{\deg N}\cap B_{\deg N-c}} \nu_{\deg
N} (x)+o(1)=\frac{\sum_{x\in A_{\deg N}\cap B_{ \deg N-k}} \nu_{\deg
N} (g)}{\sum_{x\in  L(D) \cap B_{\deg N}} \nu_{\deg N} (x)}+o(1).$$
So $\{A_{\deg N}\cap B_{\deg N-c}({\rm mod} N L(D) )\}$
 has positive upper density
relative to $\{\tilde{\nu}_N\}$. By Theorem \ref{rsinversesys},
there is a subset $A_{\deg N}\cap B_{\deg N-c}({\rm mod} NL(D))$ and
a truncated principal residue class $A$ of $L(D)$ of size $|O_C\cap
B_k|$ such that
$$A({\rm mod} N L(D) )\subseteq A_{\deg N}\cap B_{\deg N-c}({\rm mod} NL(D)).$$
Replace  $A$ by a translation if necessary, we
conclude that
$$A\subseteq A_{\deg N}\cap B_{\deg N-c}.$$ The theorem follows.
\endproof

\section{The cross-correlation of the truncated von Mangolt function} In this section
we shall establish the cross-correlation of the truncated von
Mangolt function.

The truncated von Mangolt function for the rational number field was
introduced by Heath-Brown \cite{Heath-Brown}. The truncated von
Mangolt function for the Gaussian number field was introduced by Tao
\cite{tao:gaussian}. The cross-correlation of the truncated von
Mangolt function for the rational number field were studied by
Goldston-Y{\i}ld{\i}r{\i}m in \cite{gy1-cor, gy2-cor, gy3-cor}, and
by Green-Tao in \cite{gt-ap, gt-lattice}. The cross-correlation of
the truncated von Mangolt function for the rational function field
were studied by \cite{thai}.

Let $\varphi: {\mathbb R} \to {\mathbb R}^+$ be  a smooth bump
function supported on $[-1,1]$ which equals 1 at 0, and let $R>1$ be
a parameter.
 We now define the
truncated von Mangoldt function for the function field
$K$.\begin{definition}We define the truncated \emph{von Mangoldt
function} $\Lambda_{K,R}$ of $K$ by the formula
$$
\Lambda_{K,R}(D) := \sum_{M\leq D}\mu_K(M)\varphi (\frac{\deg
M}{R}),\ \forall D\geq0,$$ where $\mu_K$ is the \emph{M\"obius
function} of $K$ defined by the formula
$$\mu_K(D)=\left\{
                      \begin{array}{ll}
                        (-1)^k, & \hbox{} D\text{ is a sum of }k\text{ distinct prime divisors},\\
                        0, & \hbox{otherwise.}
                      \end{array}
                    \right.
$$
\end{definition}
Note that $\Lambda_{K,R}(D)=1$ if $D$ is a prime divisor of degree
$\geq R$.

Let $\zeta_K(z)$ be the zeta function of $K$ defined by the formula
$$\zeta_{K}(z)=\prod_{P}\frac{1}{1-q^{-z\deg P}},\ \Re z>1,$$where
$P$ runs through the set of closed points on $C$. Write $$
\hat{\varphi}(x)= \int_{-\infty}^\infty e^t \varphi(t) e^{ixt}\
dt,$$and
$$ c_{\varphi} :=\int_{-\infty}^{+\infty}\int_{-\infty}^{+\infty} \frac{(1+iy)(1+iy')}{ (2+iy+iy')}\hat{\varphi}(y) \hat{\varphi}(y') dy
dy'.
$$
From now on, for each $r\in {\mathbb N}$, let $$ \nu_{r}
(x)=\frac{\phi_K(W) R\cdot{\rm Res}_{z=1} \zeta_K(z)}{ c_{\varphi}
q^{\deg W [K:{\mathbb F}_q(t)]}}\Lambda_{K,R}^2((Wx+\alpha) L(D)
^{-1}).$$ Here $$R=\frac{r}{8| O_C \cap B_k|2^{| O_C \cap
 B_k|}},$$ $W$ is the
product of monic irreducible polynomials of degree $\leq w:=\log\log
r$,

$\phi_K(W):=| O_C /(W))^{\times}|$, and
 $\alpha$ a number prime to $W$.

 We now prove the following.
\begin{theorem}\label{crosscorrelation}The system $\{ \nu_{r} \}$ satisfies the $k$-cross-correlation condition.
\end{theorem}
\proof Given any open compact ${\mathbb F}_q[[1/t]]$-module $I$ in
$K_{\infty}$, given any positive integers $s\leq | O_C \cap B_k|2^{|
O_C \cap B_k|},m\leq 2| O_C \cap B_k|$, and given any mutually
independent linear forms $\psi_1,\cdots,\psi_{s}$ in $m$ variables
whose coefficients are
 functions in $ O_C \cap B_{k} $, we show that
$$\frac{1}{| L(D) \cap(t^{r} I)|^m}\sum_{\stackrel{x_i\in  L(D) \cap(t^{r} I)}{i=1,\cdots,m}}\prod_{j=1}^{s}
\nu_{r} (\psi_j(x)+b_j)=1+o(1),\ r\rightarrow\infty$$ uniformly for
all   functions $b_1,\cdots,b_{s}\in
 L(D) $.

We have
$$\frac{1}{| L(D) \cap(t^{r} I)|^m}(\frac{ q^{ [K :{\mathbb F}_q(t)]\deg W
}}{\phi_K(W) R\cdot{\rm
Res}_{z=1}\zeta_K(z)})^{s}\sum_{\stackrel{x_i\in  L(D) \cap(t^{r} I)}{i=1,\cdots,m}}\prod_{j=1}^{s}
\nu_{r} (\psi_j(x)+b_j)$$$$=\sum_{{\frak d},{\frak
d}'}\omega(({\frak d}_i\cap{\frak d}_i')_{1\leq i\leq
s})\prod_{i=1}^s\mu_{K}({\frak d}_i)\mu_{K}({\frak
d}_i')\varphi(\frac{ \deg  {\frak d}_i}{ R})\varphi(\frac{ \deg
{\frak d}'_i}{ R }),$$ where ${\frak d}$ and ${\frak d}'$ run over
$s$-tuples of effective divisors on $C $, and
$$\omega(({\frak d}_i)_{1\leq i\leq s})
=\frac{|\{x\in( L(D) /( L(D-{\frak d})^m:{\frak d}_i\leq (W\psi_i(x)
+ b'_i)+D,\forall i=1,\cdots,s\}|}{({\rm N}{\frak d})^{m}}$$ with
${\frak d}={\rm l.c.m.}({\frak d}_1,\cdots,{\frak d}_s)$ and
$b'_i=Wb_i+\alpha$.

Define
$$F(t,t')=\sum_{{\frak d},{\frak d}'}\omega(({\frak
d}_j\cap{\frak d}'_j)_{1\leq j\leq
s})\prod_{j=1}^s\frac{\mu_{K}({\frak d}_j)\mu_{K}({\frak d}'_j)}
{\N({\frak d}_j)^{\frac{1+it_j}{ R }}\N({\frak
d}'_j)^{\frac{1+it'_j}{ R }}},\ t,t'\in{\mathbb R}^s,$$where ${\frak
d}$ and ${\frak d}'$ run over $s$-tuples of divisors on $ C .$

 It is easy to see that, for all $B >
0$,
$$
e^x\varphi(x) = \int_{-\sqrt{ R }}^{\sqrt{ R }} \hat{\varphi}(t)
e^{-ixt}\ dt+O( R ^{-B}).
$$
It follows that for all $B>0$,$$\frac{1}{| L(D) \cap(t^{r} I)|^m}(\frac{ q^{ [K :{\mathbb F}_q(t)]\deg W
}}{\phi_K(W) R\cdot{\rm
Res}_{z=1}\zeta_K(z)})^{s}\sum_{\stackrel{x_i\in  L(D) \cap(t^{r} I)}{i=1,\cdots,m}}\prod_{j=1}^{s}
\nu_{r} (\psi_j(x)+b_j)$$$$= \int_{[-\sqrt{ R},\sqrt{
R }]^s}\int_{[-\sqrt{ R },\sqrt{ R }]^s}
F(t,t')\hat{\varphi}(t)\hat{\varphi}(t')dtdt'$$$$+O(
R^{-B})\cdot\sum_{{\frak d},{\frak d}'}\omega(({\frak d}_j\cap{\frak
d}'_j)_{1\leq j\leq s})\prod_{i=1}^s\frac{|\mu_{K}({\frak
d}_j)\mu_{K}({\frak d}'_j)|} {\N({\frak d}_j)^{1/ R }\N({\frak
d}'_j)^{1/ R }}.
$$
Hence we are reduced to prove the following.
$$\sum_{{\frak d},{\frak d}'}\omega(({\frak
d}_j\cap{\frak d}'_j)_{1\leq j\leq
s})\prod_{j=1}^s\frac{|\mu_{K}({\frak d}_j)\mu_{K}({\frak d}'_j)|}
{\N({\frak d}_j)^{1/ R }\N({\frak d}'_j)^{1/R}}\ll R^{O_{s}(1)},$$
and, for $t,t'\in[-\sqrt{ R},\sqrt{ R }]^s$,
$$
F(t,t')=(1 + o(1))(\frac{ q^{ [K :{\mathbb F}_q(t)]\deg W
}}{\phi_K(W) R\cdot{\rm
Res}_{z=1}\zeta_K(z)})^{s}\prod_{j=1}^s\frac{(1+it_j)(1+it'_j)}{
(2+it_j+it'_j)}.$$

We prove the equality first. Applying the Chinese remainder theorem,
one can show that
$$\omega(({\frak
d}_j)_{1\leq j\leq s})=\prod_{\wp}\omega(({\frak d}_j,\wp)_{1\leq
j\leq s}),$$ where $\wp$ runs over prime divisors on $ C $. One can
also show that
$$\omega((({\frak
d}_j,\wp))_{1\leq j\leq s})=\left\{
                     \begin{array}{ll}1, & \hbox{ } \prod_{j=1}^s({\frak
d}_j,\wp)=0,\\
0,&\hbox{}  \prod_{j=1}^s({\frak d}_j,\wp)\neq0, (W)\geq\wp.
                     \end{array}
                   \right.
$$
And, if $\wp\nmid W$ and $W$ is sufficiently large, then one can
show that
$$\omega((({\frak
d}_j,\wp))_{1\leq j\leq s})\left\{
                    \begin{array}{ll}
                      =1/{\rm N}\wp, & \hbox{} \prod_{j=1}^s({\frak d}_j,\wp)=\wp\\
                      \leq1/{\rm N}\wp^2, & \hbox{}\wp^2\mid \prod_{j=1}^s({\frak
d}_j,\wp).
                    \end{array}
                  \right.
$$
It follows that
$$
F(t,t')=\prod_{\wp}\sum_{{\frak d}_j,{\frak d}'_j\mid \wp,\forall
j=1,\cdots,s}\omega(({\frak d}_j\cap{\frak d}'_j)_{1\leq j\leq
s})\prod_{j=1}^s\frac{\mu_{K}({\frak d}_j)\mu_{K}({\frak d}'_j)}
{{\rm N}{\frak d}_j^{\frac{1+it_j}{ R }}{\rm N}{{\frak
d}'_j}^{\frac{1+it'_j}{ R }}}
$$$$ =\prod_{\wp\nmid W}(1 + \sum_{j=1}^s-{\rm
N}\wp^{-1-\frac{1+it_j}{ R }} -{\rm N}\wp^{-1-\frac{1+it'_j}{\log
R}}+ {\rm N}\wp^{-1-\frac{2+it_j+it'_j}{ R }}+ O_{s}(\frac{1}{{\rm
N}\wp^2}))$$$$ =\prod_{p\nmid W}(1 + O_{s}(\frac{1}{q^{2\deg p}}))
\prod_{j=1}^s\prod_{\wp\nmid W}\frac{(1 -{\rm
N}\wp^{-1-\frac{1+it_j}{ R }})(1 -{\rm N}\wp^{-1-\frac{1+it'_j}{ R
}})}{ (1-{\rm N}\wp^{-1-\frac{2+it_j+it'_j}{ R }})}$$
$$ =(1+O(\frac{1}{ R })) \prod_{j=1}^s \frac{
\zeta_K(1+\frac{2+it_s+it'_s}{ R })}{\zeta_K(1+\frac{1+it_s}{
R})\zeta_K(1+\frac{1+it'_j}{ R })}\prod_{\wp\mid W}\frac{ (1-{\rm
N}\wp^{-1-\frac{2+it_j+it'_j}{ R }})}{(1 -{\rm
N}\wp^{-1-\frac{1+it_j}{ R }})(1 -{\rm N}\wp^{-1-\frac{1+it'_j}{ R
}})}.$$ From the estimate
$$\zeta_K(z)=\frac{{\rm Res}_{z=1}\zeta_K(z)}{z-1}+O(1), \ z\to 1,$$
and the estimate $$e^z=1+O(z),\ z\to 0,$$ we infer that
$$ F(t,t')
=(1 + O(\frac {1}{ R }))\cdot\prod_{\wp\mid
W}(1+O(\frac{\deg\wp}{\deg\wp  R^{1/2}}))\cdot$$$$(\frac{ q^{ [K
:{\mathbb F}_q(t)]\deg W }}{\phi_K(W) R \cdot{\rm
Res}_{z=1}\zeta_K(z)})^{s}\prod_{j=1}^s\frac{(1+it_j)(1+it'_j)}{
(2+it_j+it'_j)}.$$ Applying the estimate
$$\prod_{\wp\mid W}(1+\frac{\log {\rm
N}\wp}{{\rm N}\wp})=O(e^{\log^2w}),$$ we arrive at
$$ F(t,t')
=(1 + o(1))(\frac{ q^{ [K :{\mathbb F}_q(t)]\deg W }}{\phi_K(W)
R\cdot{\rm
Res}_{z=1}\zeta_K(z)})^{s}\prod_{j=1}^s\frac{(1+it_j)(1+it'_j)}{
(2+it_j+it'_j)}$$ as required.

We now turn to prove the estimate $$\sum_{{\frak d},{\frak
d}'}\omega(({\frak d}_j\cap{\frak d}'_j)_{1\leq j\leq
s})\prod_{j=1}^s\frac{|\mu_{K}({\frak d}_j)\mu_{K}({\frak d}'_j)|}
{\N({\frak d}_j)^{1/ R }\N({\frak d}'_j)^{1/\log R}}\ll
R^{O_{s}(1)}.$$ We have $$\sum_{{\frak d},{\frak d}'}\omega(({\frak
d}_j\cap{\frak d}'_j)_{1\leq j\leq
s})\prod_{j=1}^s\frac{|\mu_{K}({\frak d}_j)\mu_{K}({\frak d}'_j)|}
{\N({\frak d}_j)^{1/ R }\N({\frak d}'_j)^{1/\log
R}}$$$$=\prod_{\wp}\sum_{{\frak d}_j,{\frak d}'_j\mid \wp,\forall
j=1,\cdots,s}\omega(({\frak d}_j\cap{\frak d}'_j)_{1\leq j\leq
s})\prod_{j=1}^s\frac{1} {{\rm N}{\frak d}_j^{\frac{1}{ R }}{\rm
N}{{\frak d}'_j}^{\frac{1}{ R }}}$$$$ =\prod_{\wp\nmid W}(1 + {\rm
N}\wp^{-1-\frac{1}{ R }})^{O(1)}$$$$ =\prod_{p}(1 + q^{(-1-\frac{1}{
R })\deg p})^{O(1)}=\zeta_{{\mathbb F}_q(t)}(1+\frac{1}{
R})^{O(1)}\ll R^{O(1)}.$$
 This completes the proof of the theorem.\endproof

\section{The auto-correlation of the truncated von Mangolt function} In this section
we shall establish the auto-correlation of the truncated von Mangolt
function.

The auto-correlation of the truncated von Mangolt function for the
rational number field was studied by Goldston-Y{\i}ld{\i}r{\i}m in
\cite{gy1-cor, gy2-cor, gy3-cor}, and by Green-Tao in \cite{gt-ap,
gt-lattice}. The auto-correlation of
the truncated von Mangolt function for the rational function field
were studied by \cite{thai}.

 We now prove the following.
\begin{theorem}\label{autocorrelation}The system $\{ \nu_{r} \}$ satisfies the $k$-auto-correlation condition.
\end{theorem}
The above theorem follows from the following lemma.
\begin{lemma}Let $I$ be any open compact ${\mathbb F}_q[[1/t]]$-module in $K_{\infty}$. Then
$$\frac{1}{|(t^rI)\cap L(D) |}
\sum\limits_{x\in (t^rI)\cap L(D) }\prod_{i=1}^s
  \nu_{r} (x +
y_i) \ll \prod_{1\leq i<j\leq s}\prod_{\wp\mid(y_i-y_j)}(1+
O_{s}(\frac{1}{{\rm N}\wp}))$$ uniformly for all $s$-tuples $y\in
L(D) ^s$ with distinct coordinates.\end{lemma}

\proof We may assume that
$$\Delta:=\prod_{i\neq j}(y_i-y_j)\neq0.$$
Define
$$\omega_2(({\frak d}_i)_{1\leq i\leq s})
=\frac{|\{x\in L(D) / L(D-{\frak d}): {\frak
d}_i\leq(Wx + h_i)+D,\forall i=1,\cdots,s\}|}{{\rm
N}{\frak d}},$$ where $h_i=Wb(y)+Wy_i+\alpha$.
Then $$\frac{1}{|(t^rI)\cap L(D) |}
\sum\limits_{x\in (t^rI)\cap L(D) }(\frac{ q^{ [K :{\mathbb F}_q(t)]\deg W }}{\phi_K(W)
R})^{s}\prod_{i=1}^s
  \nu_{r} (x +
y_i)$$$$=\sum_{{\frak d},{\frak d}'}\omega_2(({\frak d}_i\cap{\frak
d}'_i)_{1\leq i\leq s})\prod_{i=1}^s\mu_{K}({\frak
d}_i)\mu_{K}({\frak d}'_i)\varphi(\frac{ \deg  {\frak d}_i}{\log
R})\varphi(\frac{ \deg  {\frak d}'_i}{ R }),$$ where ${\frak
d}$ and ${\frak d}'$ run over $s$-tuples of divisors on $ C $.
Define
$$F_2(t,t')=\sum_{{\frak d},{\frak d}'}\omega_2(({\frak
d}_i\cap{\frak d}'_i)_{1\leq i\leq
s})\prod_{j=1}^s\frac{\mu_{K}({\frak d}_j)\mu_{K}({\frak d}'_j)}
{\N({\frak d}_j)^{\frac{1+it_j}{ R }}\N({\frak
d}'_j)^{\frac{1+it'_j}{ R }}},\ t,t'\in{\mathbb R}^s,$$where
${\frak d}$ and ${\frak d}'$ run over $s$-tuples of ideals of $ O_C
.$

For all $B>0$, we have
$$\frac{1}{|(t^rI)\cap L(D) |}
\sum\limits_{x\in (t^rI)\cap L(D) }(\frac{ q^{ [K :{\mathbb F}_q(t)]\deg W }}{\phi_K(W)
R})^{s}\prod_{i=1}^s
  \nu_{r} (x +
y_i)$$$$=\int_{[-\sqrt{ R },\sqrt{ R }]^s}\int_{[-\sqrt{ R },\sqrt{ R }]^s}
F_2(t,t')\psi(t)\psi(t')dtdt'$$$$+O_B(
R^{-B})\cdot\sum_{{\frak d},{\frak d}'}\omega_2(({\frak
d}_i\cap{\frak d}'_i)_{1\leq i\leq
s})\prod_{j=1}^s\frac{|\mu_{K}({\frak d}_j)\mu_{K}({\frak d}'_j)|}
{\N({\frak d}_j)^{1/ R }\N({\frak d}'_j)^{1/ R }}.
$$
Hence we are reduced to prove the following.
$$\sum_{{\frak d},{\frak d}'}\omega_2(({\frak
d}_i\cap{\frak d}'_i)_{1\leq i\leq
s})\prod_{j=1}^s\frac{|\mu_{K}({\frak d}_j)\mu_{K}({\frak d}'_j)|}
{\N({\frak d}_j)^{1/ R }\N({\frak d}'_j)^{1/
R}}\ll R^{O_{s}(1)},$$ and, for $t,t'\in[-\sqrt{\log
R},\sqrt{ R }]^s$,
$$
F_2(t,t') \ll (\frac{ q^{ [K :{\mathbb F}_q(t)]\deg W }}{\phi_K(W)
R})^{s}\prod_{\wp\mid\Delta,\wp\nmid W}(1+ O_{s}(\frac{1}{{\rm
N}\wp}))\prod_{j=1}^s\frac{(1+|t_j|)(1+|t'_j|)}{
(2+|t_j|+|t'_j|)}.$$

We prove the second inequality but omit the proof of first one.
Applying the Chinese remainder theorem, one can show that
$$\omega_2(({\frak
d}_i)_{1\leq i\leq s})=\prod_{\wp}\omega_2(({\frak d}_i,\wp)_{1\leq
i\leq s}).$$ One can also show that
$$\omega_2((({\frak
d}_i,\wp))_{1\leq i\leq s})=\left\{
                     \begin{array}{ll}1, & \hbox{ } \prod_{i=1}^s({\frak
d}_i,\wp)=(1),\\
0,&\hbox{}  \prod_{i=1}^s({\frak d}_i,\wp)\neq(1), \wp|W.
                     \end{array}
                   \right.
$$
And, if $\wp\nmid W$ and $w$ is sufficiently large, then one can
show that
$$\omega_2((({\frak
d}_i,\wp))_{1\leq i\leq s})\left\{
                    \begin{array}{ll}
                      =1/{\rm N}\wp, & \hbox{} \prod_{i=1}^s({\frak d}_i,\wp)=\wp\\
                     =0, & \hbox{}\wp^2\mid \prod_{i=1}^s({\frak
d}_i,\wp),\wp\nmid\Delta,\\
\leq1/{\rm N}\wp, & \hbox{}\wp^2\mid \prod_{s\in S}({\frak
d}_s,\wp),\wp\mid\Delta.
                    \end{array}
                  \right.
$$
It follows that
$$
F_2(t,t')=\prod_{\wp}\sum_{{\frak d}_i,{\frak d}'_i\mid \wp,\forall
i=1,\cdots,s}\omega_2(({\frak d}_i\cap{\frak d}'_i)_{1\leq i\leq
s})\prod_{j=1}^s\frac{\mu_{K}({\frak d}_j)\mu_{K}({\frak d}'_j)}
{{\rm N}{\frak d}_s^{\frac{1+it_j}{ R }}{\rm N}{{\frak
d}'_j}^{\frac{1+it'_j}{ R }}}
$$$$ =\prod_{\wp\nmid W\Delta}(1 + \sum_{j=1}^s-{\rm
N}\wp^{-1-\frac{1+it_j}{ R }} -{\rm N}\wp^{-1-\frac{1+it'_j}{
R}}+ {\rm N}\wp^{-1-\frac{2+it_j+it'_j}{ R }})\prod_{\wp\nmid
W,\wp\mid\Delta}(1+ O_{s}(\frac{1}{{\rm N}\wp}))$$$$\ll
\prod_{\wp\mid\Delta,\wp\nmid W}(1+ O_{s}(\frac{1}{{\rm
N}\wp}))\prod_{j=1}^s\prod_{\wp\nmid W\Delta}\frac{(1 -{\rm
N}\wp^{-1-\frac{1+it_j}{ R }})(1 -{\rm
N}\wp^{-1-\frac{1+it'_j}{ R }})}{ (1-{\rm
N}\wp^{-1-\frac{2+it_j+it'_j}{ R }})}$$
$$ \ll (\frac{ q^{[K :{\mathbb F}_q(t)] \deg W }}{\phi_K(W)})^{s}\prod_{\wp\mid\Delta,\wp\nmid W}(1+
O_{s}(\frac{1}{{\rm N}\wp})) \prod_{j=1}^s \frac{
\zeta_K(1+\frac{2+it_j+it'_j}{ R })}{\zeta_K(1+\frac{1+it_j}{\log
R})\zeta_K(1+\frac{1+it'_j}{ R })}$$$$
 \ll (\frac{ q^{ [K :{\mathbb F}_q(t)] \deg W}}{\phi_K(W) R })^{s}\prod_{\wp\mid\Delta,\wp\nmid W}(1+
O_{s}(\frac{1}{{\rm N}\wp}))\prod_{j=1}^s\frac{(1+|t_j|)(1+|t'_j|)}{
(2+|t_j|+|t'_j|)}.$$
 This completes the proof of the lemma.\endproof
\section{Proof of the main theorem}
In this section we prove Theorem \ref{main}.
We begin with the following lemma.
\begin{lemma}There is a positive constant $c_K$ such that
every principal divisor on $C$ is the divisor of a function $\xi$
satisfying
$${\rm ord}_P(\xi)\geq \frac{-\deg(\xi)}{[K:{\mathbb
F}_q(t)]}-c_K,\ \forall P\mid\infty.$$
\end{lemma} \proof
Let $\varepsilon_1,\cdots,\varepsilon_r$ be a system of fundamental
units of $O_K$. For each $i=1,\cdots,r$, set
$$y_i:=({\rm ord}_P(\varepsilon_i))_{P\mid\infty}.$$ It is known that the
vectors $y_1,\cdots,y_r$ are linearly independent. Let $\xi$ be a
nonzero function in $K$. We have
$$\sum_{P\mid\infty}[K_P:{\mathbb F}_q((1/t))]({\rm ord}_P(\xi)+
\frac{\deg(\xi)}{[K:{\mathbb F}_q(t)]})=0.$$ It follows that the
vector
$$({\rm ord}_P(\xi)+ \frac{\deg(\xi)}{[K:{\mathbb
F}_q(t)]})_{P\mid\infty}$$ as well as the vectors $y_1,\cdots, y_r$
are orthogonal to the vector $([K_P:{\mathbb
F}_q((1/t))])_{P\mid\infty}$. So
$$({\rm ord}_P(\xi)+ \frac{\deg(\xi)}{[K:{\mathbb
F}_q(t)]})_{P\mid\infty} =a_1y_1+\cdots+a_ry_r.$$ Multiplying $\xi$
by a function in $O_K^{\times}$ if necessary, we may assume that
$0\leq a_i< 1$. It follows that, for every $P\mid\infty$,
$${\rm ord}_P(\xi)\geq \frac{-\deg(\xi)}{[K:{\mathbb
F}_q(t)]}-c_K.$$ The lemma now follows.
\endproof
For each $r\in {\mathbb N}$, and for each $\alpha\in L(D) $ with $(\alpha,W
L(D) )= L(D) $, set
$$A_{r,\alpha}=\{x\in L(D)\cap B_r \mid (Wx+\alpha)+D \text{ is prime}\}.$$
By Theorem \ref{rs}, Theorem \ref{main} follows from the following
theorem.
\begin{theorem}Let $c$ be a sufficiently large positive constant
depending only on $k$, $C$ and $D$. For each $r\in {\mathbb N}$,
there is a number $\alpha_r\in(t^{\deg W}G)\cap L(D) $
 with $((\alpha_r)+D,W)=1 $ such that the system $\{A_{r,\alpha_r}\cap
B_{r-c}\}$ has positive upper density relative to $\{ \nu_{r}
\}$.\end{theorem} \proof By the above lemma, there is a positive
constant $c_K$ such that every principal divisor on $C$ is the
divisor of a function $\xi$ satisfying
$${\rm ord}_P(\xi)\geq \frac{-\deg(\xi)}{[K:{\mathbb
F}_q(t)]}-c_K,\ \forall P\mid\infty.$$ It follows that, for any
divisor ${\frak n}\in[D]$, there is an element $x\in L(D) $ such
that ${\frak n}=(x)+D$ and that $${\rm ord}_P(x)\geq \frac{\deg
D-\deg{\frak n}}{[K:{\mathbb F}_q(t)]}-c_K,\ \forall P\mid\infty.$$
In particular, for each $r\in {\mathbb N}$, and for any prime
divisor $\wp\in[ D]$ satisfying $(\wp,W)=1 $ and $\deg\wp<\deg D+[K
:{\mathbb F}_q(t)](r-c-c_K-\deg W)$, there is a function $x\in
A_{r,\alpha}\cap B_{r-c}$, and a function $\alpha\in(t^{\deg
W}G)\cap L(D) $ with $((\alpha)+D,W)=1$ such that $\wp=(Wx+\alpha)
+D$. So
$$\sum_{\stackrel{((\alpha)+D,W)=1 }{\alpha\in(t^{\deg W}G)\cap L(D) }}\sum_{x\in A_{r,\alpha}\cap
B_{r-k}}\Lambda_{K,R}^2((Wx+\alpha)+D
)$$$$\geq\sum_{\stackrel{\wp\in[D],(\wp,W)=1}{\deg\wp<\deg D+[K
:{\mathbb F}_q(t)](r-c-c_K-\deg W)}}\Lambda_{K,R}^2(\wp)\gg\frac{
c_{\varphi} q^{\deg W [K:{\mathbb F}_q(t)]}}{ R\cdot{\rm Res}_{z=1}
\zeta_K(z)}\cdot |L(D)\cap B_r|.$$ The theorem now follows by the
pigeonhole principle.
\endproof


\begin{thebibliography}{99}
\bibitem[GY1]{gy1-cor} D. Goldston and C.Y. Y{\i}ld{\i}r{\i}m, \emph{Higher correlations of divisor sums related to primes, I: Triple correlations,} Integers \textbf{3} (2003) A5, 66pp.

\bibitem[GY2]{gy2-cor} D. Goldston and C.Y. Y{\i}ld{\i}r{\i}m, \emph{Higher correlations of divisor sums related to primes, III: $k$-correlations,} preprint (available at AIM preprints)

\bibitem[GY3]{gy3-cor} D. Goldston and C.Y. Y{\i}ld{\i}r{\i}m, \emph{Small gaps between primes, I,} preprint.
\bibitem[Gow]{gow-sze}
T. Gowers, \emph{Hypergraph regularity and the multidimensional Szemer\'edi theorem}, preprint.
\bibitem[GT1]{gt-ap}
B. Green, T. Tao, \emph{The primes contain arbitrarily long arithmetic progressions}, Ann. Math. 167 (2008), 481-547.
\bibitem[GT2]{gt-lattice}
B. Green, T. Tao, \emph{Linear equations in primes}, Ann. Math. 171
(2010), 1753-1850.
\bibitem[HB]{Heath-Brown}
D. R. Heath-Brown, \emph{The ternary Goldbach problem}, Rev. Mat.
Iberoamericana, 1 (1985), 45-59.

\bibitem[Tao]{tao:gaussian}
T. Tao, \emph{The Gaussian primes contain arbitrary shaped constellations}, J. d. Analyse Mathematique 99 (2006), 109-176.
\bibitem[Thai]{thai}
Thai Hoang Le, \emph{Green-Tao theorem in function field}, Preprint.


\end{thebibliography}
\end{document}